\documentclass[12pt,a4paper]{article}
\usepackage[]{amsmath,amssymb,amsthm}
\usepackage{epsfig}

\newcommand{\R}{\mathbb{R}}
\newcommand{\N}{\mathbb{N}}

\newcommand{\Z}{\mathbb{Z}}

\newtheorem{theorem}{Theorem}[section]
\newtheorem{proposition}[theorem]{Proposition}%[section]
\newtheorem{corollary}[theorem]{Corollary}%[section]
\newcommand{\cao}{\c{c}\~ao}

\newcommand{\dpt}{\displaystyle}

\begin{document}

\begin{center}
{\large {\bf The Discrete Markus-Yamabe Problem for Symmetric Planar Polynomial Maps}}\\
\mbox{} \\
\begin{tabular}{ccc}
{\bf Bego\~na Alarc\'on$^{\mbox{a}}$} & {\bf Sofia B.\ S.\ D.\ Castro$^{\mbox{b},\dagger}$} & {\bf Isabel S.\ Labouriau$^{\mbox{a}}$} 
\end{tabular}
\end{center}

\bigbreak
\noindent {\small $^{\mbox{a}}$} Centro de Matem\'atica da Universidade do Porto, Rua do Campo Alegre 687, 4169-007 Porto, Portugal.

\noindent {\small $^{\mbox{b}}$} Faculdade de Economia do Porto, Rua Dr. Roberto Frias, 4200-464 Porto, Portugal.
 
\noindent {\small $\dagger$} Corresponding author:  sdcastro@fep.up.pt; phone +351 225 571 151; fax +351 225 505 050.

\bigbreak

\begin{center} To the memory of Carlos Gutierrez \end{center}
\begin{abstract}
We probe deeper into the Discrete Markus-Yamabe Question for polynomial planar maps and 
%find a 
into the
normal form for those maps which answer this question in the affirmative. Furthermore, in a symmetric context, we show that the only nonlinear equivariant polynomial maps providing an affirmative answer to the Discrete Markus-Yamabe Question are those possessing $\Z_2$ as their group of symmetries.
We use this to 
establish two new tools which give information about the spectrum of a planar polynomial map.
\end{abstract}

\bigbreak


\bigbreak

\section{Introduction}

The Discrete Markus-Yamabe Question, DMYQ($n$), in dimension $n$ was formulated by
Cima {\em et al.} \cite{CGM} as follows:
\medbreak
\noindent{\em
%\begin{question}[DMYQ($n$)]
{\bf [DMYQ($n$)]}
Let $F:\R^n\longrightarrow\R^n$ be a $C^1$ map such that $F(0)=0$ and for any $x\in\R^n$, $ JF(x)$ has all its eigenvalues with modulus less than one. 
Is it true that $0$ is a global attractor for the discrete dynamical system generated by $F$?
%\end{question}
}
\medbreak\noindent
These authors have found sufficient conditions for planar maps to provide an affirmative answer to this question. 
We proceed with the study in dimension $2$,
since this is the only interesting dimension: the answer is negative for higher dimensions, see Cima {\em et al.} \cite{CvdEGHM} for examples in dimensions higher than $3$ and van den Essen and Hubbers \cite{vdEH} for dimensions higher than $4$, and is affirmative in dimension $1$.

%Since dimension $2$ is the only interesting dimension (the answer is negative for higher dimensions, see Cima {\em et al.} \cite{CvdEGHM} for examples in dimensions higher than $3$ and van den Essen and Hubbers \cite{vdEH} for dimensions higher than $4$, and affirmative in dimension $1$), we proceed with the study of maps satisfying the hypothesis of the DMYQ($2$). 

An attentive look at the proof of Theorem B in Cima {\em et al.} \cite{CGM} produces a more explicit description of polynomial maps satisfying sufficient conditions for an affirmative answer to the DMYQ($2$), leading to a normal form for such maps, as in Chamberland~\cite{Chamberland}. In particular
%, it follows that the quotient of the linear parts must be constant and 
this may be used  for  testing a  map for eigenvalues outside the unit disk.

After having established the normal form for maps that answer the DMYQ($2$) in the affirmative, we look at the symmetric setting. We formulate the Symmetric Discrete Markus-Yamabe Question, SDMYQ($n$), as follows:
\medbreak
\noindent{\em
%\begin{question}[DMYQ($n$)]
{\bf [SDMYQ($n$)]}
Let $F:\R^n\longrightarrow\R^n$ be a $C^1$ map such that $F(0)=0$ and for any $x\in\R^n$, $ JF(x)$ has all its eigenvalues with modulus less than one. 
Suppose that the symmetries of $F$ form a nontrivial compact subgroup of $O(n)$.
Is it true that $0$ is a global attractor for the discrete dynamical system generated by $F$?
%\end{question}
}
\medbreak

Note that a counterexample to the SDMYQ($2$) is given in  \cite[theorem D]{CGM}, where $F$ is a rational map, and the symmetries of $F$ constitute the group $\Z_4$.

We address this question for $n=2$ when $F$ is polynomial.
We find that only when the group of symmetries of the map is  $\Z_2$
% = \langle -Id \rangle$ (the group of order two generated by minus the identity), 
(a group of order two),
 can a nonlinear polynomial map 
 provide an affirmative answer to the DMYQ($2$) and {\sl a fortiori} to the SDMYQ($2$).
In fact, we show that this is the only symmetry group compatible with the hypotheses of the  DMYQ($2$).
 This is then used as a test for the existence of expanding eigenvalues in symmetric polynomial maps.

\section{Normal Forms for Planar Polynomial Maps}

We look deeper into the admissible form of polynomial maps which provide an affirmative answer to the DMYQ($2$).

\begin{theorem}[Normal Form]\label{poly_lemma}
Let $F:\R^2 \rightarrow \R^2$ be a polynomial map such that $F(0)=0$ and all the eigenvalues of $JF(x,y)$ have modulus smaller than one for all $(x,y) \in \R^2$. Then $F(x,y) = B (x,y)^{T} + u^2 p(u)(\alpha, \beta)^ {T}$, where $B$ is a real matrix, $\alpha, \beta \in \R$, $p$ is a real polynomial and $u=ax+by$ for $a,b \in \R$. 
%with real eigenvalues
\end{theorem}

\begin{proof}
Theorem B in Cima {\em et al.} \cite{CGM} proves that the condition on the eigenvalues of $JF$ implies that $F$ is obtained by an affine transformation from a triangular map. The assumption that the origin is a fixed-point allows us to work with linear instead of affine transformations.
Furthermore, the triangular map is such that the off-diagonal terms can be described by a polynomial in one variable alone. Therefore, from the triangular map $G(u,v)=K(u,v)^ {T}+ (0, u q(u))^ {T}$ with $K$ a real diagonal matrix and $q$ a real polynomial, we obtain by a linear change of coordinates $L$,

%\begin{eqnarray*}
%P(u,v)=k_1u & \mbox{  and   } & Q(u,v)=k_2v+p(u)
%\end{eqnarray*}
%for $k_1,k_2 \in \R$, we obtain using a linear transformation $L$
$$
F(x,y) = L^{-1} G(L(x,y)).
$$
Let 
$$
L=\left( \begin{array}{cc}
a & b \\
c & d
\end{array} \right)
$$ 
then for $u=ax+by$ we get

$$F(x,y)=L^{-1}\left(KL\left( \begin{array}{c}
x \\
y
\end{array} \right)+ \left( \begin{array}{c}
0 \\
uq(u)
\end{array} \right)\right)= $$
$$=L^{-1}KL\left( \begin{array}{c}
x \\
y
\end{array} \right)+uq(u)L^{-1}\left( \begin{array}{c}
0 \\
1
\end{array} \right)$$
Let $A=L^{-1}KL$,  then $A$ is a matrix with real eigenvalues. Consider now 
\begin{equation}
\label{alfabeta}(\alpha,\beta)^{T}=L^{-1}(0,1)^T=\dfrac{1}{ad-bc} \left(\begin{array}{c}-b\\ a\end{array}\right),
\end{equation}
 then we can rewrite $F$ as follows:

$$F(x,y) = A \left( \begin{array}{c}
x \\
y
\end{array} \right) + uq(u)\left( \begin{array}{c}
\alpha \\
\beta
\end{array} \right)$$

The proof follows taking $$B\left( \begin{array}{c}
x \\
y
\end{array} \right)=A\left( \begin{array}{c}
x \\
y
\end{array} \right)+uq(0)\left( \begin{array}{c}
\alpha \\
\beta
\end{array} \right)
\qquad\mbox{and}\qquad 
p(u)=\frac{q(u)-q(0)}{u} \ .$$ 
%with $|L|$ the determinant of $L$. 
%\begin{eqnarray*}
%F(x,y) & = & \frac{1}{|L|} L^{-1} G(a_{11}x+a_{12}y,a_{21}x+a_{22}y) \\
% & = & \frac{1}{|L|} (a_{22}k_1(a_{11}x+a_{12}y) - a_{12}(a_{21}x+a_{22}y)  -a_{12}p(a_{11}x+a_{12}y), \\
% & &  -a_{21} k_1a_{11}x+a_{12}y +a_{11}k_2(a_{21}x+a_{22}y)  +a_{11}p(a_{11}x+a_{12}y)).
%\end{eqnarray*}
\end{proof}

Theorem \ref{poly_lemma} is better used for identifying which maps do not provide an affirmative answer to the DMYQ($2$). In fact, while it may not be easy to recognize an admissible form by looking at a map, it will be straighforward to assert that, for instance, $F(x,y)=(\frac{x}{2}+y^2, \frac{y}{3}+y^3)$ will not provide an affirmative answer to the DMYQ($2$). This is because the non-linear polynomials in the first and second coordinate have different degrees.

This provides a criterion for studying the spectrum of a polynomial planar map as stated in the following:

\begin{corollary}\label{Coro1}
Let $F:\R^2\rightarrow\R^2$ be a polynomial map. If the quotient of the nonlinear parts of the coordinates of $F$ is not constant, then there exists a point in $\R^2$ where the jacobian of $F$ has an eigenvalue outside the unit disk.
\end{corollary}

\section{Symmetric Planar Polynomial Maps}

In the context of symmetric maps some further results may be obtained. As usual, the reference for the symmetric context is the book by Golubitsky {\em et al.} \cite{GSS}. Assume for the rest of this section that $F:\R^2 \rightarrow \R^2$ has a compact Lie group $\Gamma$ as its group of linear symmetries. That is to say that $\Gamma$ is the largest group such that for all $(x,y) \in \R^2$ and all $\gamma \in \Gamma$ we have
\begin{equation}\label{equivariant}
F(\gamma \cdot(x,y)) = \gamma\cdot F(x,y).
\end{equation}
We always assume nontrivial groups and actions.

We single out two possible group elements of $\Gamma$. These are represented by $\kappa$ and $\zeta_n $ and act on elements of the plane as
\begin{eqnarray*}
\kappa\cdot (x,y)^{T} & = & (x, -y)^{T} \\
\zeta_n \cdot (x,y)^{T} = e^{2\pi i/n}. (x,y)^{T} & = & (x\cos{\frac{2\pi}{n}}-y\sin{\frac{2\pi}{n}}, x\sin{\frac{2\pi}{n}}+y\cos{\frac{2\pi}{n}})^{T},
\end{eqnarray*}
where $n \in \N$.

Note that any reflection may be written as $\kappa$ above in suitable coordinates.

\begin{proposition}
Let $\Gamma$ be a compact Lie group acting on $\R^2$. Assume $\Gamma$ is the symmetry group of a polynomial map $F$.
\begin{itemize}
	\item[(i)]  If $\kappa \in \Gamma$ then $F$ does not answer the DMYQ($2$) in the affirmative unless $F$ is of the form:
	$$
	F(x,y)=\left(\begin{array}{cc}d_1&0\\0&d_2\end{array}\right)\left(\begin{array}{c}x\\ y\end{array}\right)+
	y^2p(y^2)\left(\begin{array}{c}1\\ 0\end{array}\right)\ .
	$$
	\item[(ii)]  If $\zeta_n  \in \Gamma$ for $n \geq 3$ then $F$ does not answer the DMYQ($2$) in the affirmative unless $F$ is linear.
	Moreover,  the linear part of $F$ is either a homothety or a rotation matrix.
\end{itemize}
\end{proposition}

\begin{proof}
From Theorem \ref{poly_lemma} we know that, in order to satisfy the hypotheses of the DMYQ($2$), we must have
$F(x,y)=B(x,y)^T+ (\alpha u^2 p(u), \beta u^2 p(u))^T$.
Then $F$ is $\gamma$-equivariant if and only if both $B$ and the nonlinear part satisfy \eqref{equivariant}.
 We then write $N(x,y) = (\alpha r(u), \beta r(u))$ with $r(u)=u^2p(u)$. 
 The proof proceeds in the following two steps:
 \medbreak
%\begin{itemize}
%	\item[(i)] 
	\noindent{\bf(i)} 
	If $\kappa \in \Gamma$, then $B\cdot \kappa=\kappa\cdot B$ if and only if $B$ is a diagonal matrix.
	Furthermore, we must have
	\begin{eqnarray*}
	N(\kappa. (x,y)) & = & (\alpha r(ax-by), \beta r(ax-by))^{T} \\
	& = & (\alpha r(ax+by), - \beta r(ax+by))^{T} = \kappa. N(x,y),
	\end{eqnarray*}
	from which it follows that, if both $\alpha$ and $\beta$ are nonzero, then
	$$
	\left\{ \begin{array}{l}
	r(ax+by) = r(ax-by) \\
	r(ax+by) = -r(ax-by)\ ,
	\end{array} \right. 
	$$
	which is to say that $r(ax-by)=-r(ax-by)$ and thus $r(u)=-r(u)$ for all $u \in \R^2$. Hence, $r$ is identically zero.
	
	If $\alpha=0$, and $\beta\neq 0$ then,  from \eqref{alfabeta} in the proof of Theorem~\ref{poly_lemma}, $b=0$.
	We then have 
	$$
	r(ax-by)=-r(ax+by)\qquad\Leftrightarrow\qquad
	r(ax)=-r(ax)
	$$
	meaning that $r$ is identically zero.
	
	If $\beta=0$, and $\alpha\neq 0$ then,  from \eqref{alfabeta} in the proof of Theorem~\ref{poly_lemma}, $a=0$.
	We then have 
	$$
	r(ax-by)=r(ax+by)\qquad\Leftrightarrow\qquad
	r(-by)=r(by)
	$$
	meaning that $r$ is an even polynomial in $y$. 
	
	If both $\alpha$ and $\beta$ are zero, the result holds trivially.
	\medbreak

%	\item[(ii)] 
	\noindent{\bf(ii)} 
	If $\zeta_n  \in \Gamma$ for $n \geq 3$, (\ref{equivariant}) implies that 
%	\begin{eqnarray*}
%	N(\zeta_n \cdot (x,y)) & = & 
%	(\alpha r(a(x\cos{\frac{2\pi}{n}}-y\sin{\frac{2\pi}{n}}) + 
%	b(x\sin{\frac{2\pi}{n}}+y\cos{\frac{2\pi}{n}})), \\
%	& & \quad \beta r (a(x\cos{\frac{2\pi}{n}}-y\sin{\frac{2\pi}{n}}) + b(x\cos{\frac{2\pi}{n}}-y\sin{\frac{2\pi}{n}}))) 
%	\end{eqnarray*}
$$
	N(\zeta_n \cdot (x,y)) =
	\left( \begin{array}{c}\dpt
	\alpha r\left(a(x\cos{\frac{2\pi}{n}}-y\sin{\frac{2\pi}{n}}) + 
	b(x\sin{\frac{2\pi}{n}}+y\cos{\frac{2\pi}{n}})\right) \\
	{}\\ \dpt
	\beta r \left(a(x\cos{\frac{2\pi}{n}}-y\sin{\frac{2\pi}{n}}) + 
	b(x\sin{\frac{2\pi}{n}}+y\cos{\frac{2\pi}{n}})\right)
	\end{array}\right) 
$$
	must be equal to
	
%	\begin{eqnarray*}
%	 \zeta_n \cdot N (x,y)
%	 & = & (\alpha r(ax+by)\cos{\frac{2\pi}{n}} - \beta r(ax+by)\sin{\frac{2\pi}{n}}, \\
%	 & & \mbox{   } \; \; \; \alpha r(ax+by)\sin{\frac{2\pi}{n}} + \beta r(ax+by)\cos{\frac{2\pi}{n}}) .
%	\end{eqnarray*}
$$
	 \zeta_n \cdot N (x,y)=\left(\begin{array}{c}\dpt
	 \alpha r(ax+by)\cos{\frac{2\pi}{n}} - \beta r(ax+by)\sin{\frac{2\pi}{n}}\\ 
	 {}\\ \dpt
	\alpha r(ax+by)\sin{\frac{2\pi}{n}} + \beta r(ax+by)\cos{\frac{2\pi}{n}} 
	\end{array}\right) .
$$
	We therefore must have
	$$
	\left\{ \begin{array}{l}
	\alpha r(a(x\cos{\frac{2\pi}{n}}-y\sin{\frac{2\pi}{n}}) + b(x\sin{\frac{2\pi}{n}}+y\cos{\frac{2\pi}{n}})) = \\
	\qquad= \alpha r(ax+by)\cos{\frac{2\pi}{n}} - \beta r(ax+by)\sin{\frac{2\pi}{n}} \\
	\mbox{} \\
	\beta r(a(x\cos{\frac{2\pi}{n}}-y\sin{\frac{2\pi}{n}}) + b(x\sin{\frac{2\pi}{n}}+y\cos{\frac{2\pi}{n}})) = \\
	\qquad= \alpha r(ax+by)\sin{\frac{2\pi}{n}} + \beta r(ax+by)\cos{\frac{2\pi}{n}}\ .
	\end{array} \right. 
	$$
	If one of either $\alpha$ or $\beta$ is zero, we observe that $r$ is identically zero since $n\geq 3$.
	Otherwise,  after some simplification, we obtain 
	$$
	-\beta^2r(ax+by) = \alpha^2 r(ax+by)
	$$
	and again we see that $r$ must be identically zero.
%\end{itemize}
\end{proof}

The following result finishes our description of planar polynomial maps that provide an affirmative answer to the DMYQ($2$).

\begin{theorem} \label{teoGuapo}
A nonlinear equivariant polynomial map satisfying DMYQ(2) can only have $\Gamma = \Z_2$ as its symmetry group.
%Let $F:\R^2 \rightarrow \R^2$ be a nonlinear polynomial map such that $F(0)=0$ and all the eigenvalues of $JF(x,y)$ have modulus smaller than one for all $(x,y) \in \R^2$. If $\Gamma$, a nontrivial compact Lie group, is the symmetry group of $F$ then $\Gamma = \Z_2 = \langle -Id \rangle$.
\end{theorem}

\begin{proof}
It is known (see, for instance, Golubitsky {\em et al.} \cite{GSS}, XII \S 1 (c)) that every compact Lie group in $GL(2)$ can be identified with a subgroup of the orthogonal group $O(2)$. 
The compact subgroups of  $O(2)$ that do not contain a rotation $\zeta_n$, $n\geq 3$, are, in suitable coordinates, the trivial subgroup generated by the identity; $\Z_2$, generated by either $\kappa$ or minus the identity; and $\Z_2\oplus\Z_2$
This last group contains the two reflections $\kappa$ and $-\kappa$. 
Therefore if $F$ has these symmetries, it must satisfy both
$$
F(x,y)=\left(\begin{array}{cc}d_1&0\\0&d_2\end{array}\right)\left(\begin{array}{c}x\\ y\end{array}\right)+
	\left(\begin{array}{c}r(y^2)\\ 0\end{array}\right)
$$
and
$$
F(x,y)=
	\left(\begin{array}{cc}d_1&0\\0&d_2\end{array}\right)\left(\begin{array}{c}x\\ y\end{array}\right)+
	\left(\begin{array}{c}0\\ \tilde{r}(x^2)\end{array}\right)
$$
and therefore, $r=\tilde{r}=0$.
Since we are assuming $\Gamma$ to be nontrivial, the proof is finished.
\end{proof}

We end this note with the following example: the lowest order (and perhaps simplest) nonlinear polynomial map whose symmetry group is $\Z_4$ is of the form $F(x,y) = (\alpha x-\beta y^3, \alpha y +\beta x^3)$.
By Theorem~\ref{teoGuapo} this map cannot answer the  DMYQ($2$) in the affirmative.
Indeed, 
it is clear either by direct computation or by applying Lemma 1.1 in \cite{CGM} that the eigenvalues of the jacobian of $F$ are not all inside the unit disk. 
 %Therefore, this map does not provide an affirmative answer to the DMYQ($2$).
 
 In fact Theorem~\ref{teoGuapo} together with Theorem~B of \cite{CGM} 
lead to a second criterion for studying the spectrum of a polynomial planar map as follows:

\begin{corollary}\label{Coro1}
Let $F:\R^2\rightarrow\R^2$ be a polynomial map. If $F$ has a non-trivial symmetry group different from $ \Z_2$, then there exists a point in $\R^2$ where the jacobian of $F$ has an eigenvalue outside the unit disk.
\end{corollary}

\bigbreak

\paragraph{Acknowledgements:} 
The research of all authors at Centro de Matem\'atica da Universidade do Porto (CMUP)
 had financial support from 
 the European Regional Development Fund through the programme COMPETE and
 from  the Portuguese Government through the Funda\cao{} para
a Ci\^encia e a Tecnologia (FCT) under the project PEst-C/MAT/UI0144/2011.
%, Portugal,
%  through the programs
%POCTI and POSI  with European Union and national funding.
B. Alarc\'on was also supported from Programa Nacional de Movilidad de Recursos Humanos of the Plan Nacional de I+D+I 2008-2011 of the Ministerio de Educaci\'on (Spain) and grant MICINN-08-MTM2008-06065 of the Ministerio de Ciencia e Innovaci\'on (Spain).

\end{document}